\documentclass[11pt]{amsart}
\usepackage[usenames,dvipsnames]{pstricks}
\usepackage[mathscr]{eucal}
\usepackage{amsfonts,amsmath,amssymb,amsthm,amscd
	,amsxtra,textcomp, stmaryrd}
\usepackage{enumerate,verbatim}
\usepackage[all,2cell,ps]{xy}
\usepackage[notcite, notref]{}
\usepackage[pagebackref]{hyperref}
\usepackage{calc,color}
\usepackage{wasysym}
\usepackage{mathptmx}
\theoremstyle{plain} 

\newtheorem{thm}{Theorem}[section]

\newtheorem{rmk}[thm]{Remark}
\newtheorem{prop}[thm]{Proposition}

\newtheorem{cor}[thm]{Corollary}
\newtheorem{lem}[thm]{Lemma}

\numberwithin{equation}{section}
\newcommand{\Ann}{\mbox{Ann}\,}

\newcommand{\Hom}{\mbox{Hom}\,}
\newcommand{\Ext}{\mbox{Ext}\,}

\newcommand{\gr}{\mbox{grade}\,}

\newcommand{\depth}{\mbox{depth}\,}
\renewcommand{\dim}{\mbox{dim}\,}

\newcommand{\id}{\mbox{inj.dim}\,}
\newcommand{\pd}{\mbox{proj.dim}\,}

\newcommand{\h}{\mbox{ht}\,}

\newcommand{\E}{\mbox{E}}

\newcommand{\fm}{\mathfrak{m}}

\newcommand{\bQ}{\mathbb{Q}}

\begin{document}
	\title{A STUDY OF STRONGLY COHEN-MACAULAY IDEALS BY DELTA INVARIANT}
	
	\author[M. T. Dibaei]{Mohammad T. Dibaei$^1$} 
	\author[Y. Khalatpour]{Yaser Khalatpour$^2$}
	
	\address{$^{1, 2}$ Faculty of Mathematical Sciences and Computer,
		Kharazmi University, Tehran, Iran.}
	
	\email{dibaeimt@ipm.ir} \email{yaserkhalatpour@gmail.com}


	\maketitle
	\begin{abstract} 
		Let $R$ be a Cohen-Macaulay local ring, $I$ a strongly Cohen-Macaulay ideal of $R$. We show that $R/\Ann_R(I)$ is a maximal Cohen-Macaulay $R$-module by means of the delta-invariant. Also it is shown that there exists  a Cohen-Macaulay ideal $J$ of $R$ such that  the $\delta_{R/J}$-invariant of  all Koszul homologies of $R$ with respect $I$  is zero. 	
	\end{abstract}

	\section{Introduction}
	 Throughout the paper, $(R, \fm, k)$ is assumed to be a Cohen-Macaulay local ring with maximal ideal $\fm$ and the residue field $k$. The delta invariant of a finite (i.e. finitely generated) $R$-module has been defined by Maurice Auslander  \cite{AB}.  For a finite module $M$ over a local ring $R$, the  $\delta$-invariant of $M$ is defined as $\delta_{R}\left( M\right)=\mu_R(M/M^\text{cm})$, where $M ^\text{cm}$ is the sum of all submodules $\phi(L) \subseteq M $, $L $ ranges over all maximal Cohen-Macaulay $R$-modules with no non-zero free direct summands, $\phi $ ranges over all $R$--linear homomorphisms from $L$ to $M $ and $\mu_R(-)$ stands for the minimum number of generators. It follows that  $\delta_{R}(M)\leq\ \mu_R(M) $. The theory of  $\delta$-invariant  is extensively studied for Cohen-Macaulay modules (see \cite{Y}, \cite{AM}. \cite{AB}, \cite{LW}). There are some papers studying conditions for an $R$ module $M$ to have the
	 property $\delta_R (M) = 0$ ( see \cite[Theorem 2.7]{D1}, \cite[Theorem 2.2]{Y1} ). This paper is devoted to provide such conditions (see Theorem \ref{3}, Proposition \ref{6}, Theorem\ref{15} ).
	  
	 An ideal $I$ is called strongly Cohen-Macaulay if $H_i(I, R)$ is Cohen--Macaulay for all $i$, where $H_i(I, R)$'s are the homology modules  of the Koszul complex of $R$ with respect to ideal  $I$. Note that a strongly Cohen-Macaulay ideal $I$ is  Cohen-Macaulay that is $R/I$ is Cohen-Macaulay as $R$-module.
	 
	 	The notion of linkage of ideals in commutative algebra is invented by Peskine and Szpiro \cite{PS}. Two ideals $I$ and $J$ in a Cohen-Macaulay local ring $R$ are said to be linked if there is a regular sequence $\underline{a}$ in their intersection such that $I = (\underline{a}) :_R J$ and $J = (\underline{a}) :_R I$. They have shown that the Cohen-Macaulay-ness property
	 is preserved under linkage over Gorenstein local rings and provided a counter example to show that  the above result is no longer true if the base ring is Cohen-Macaulay but not Gorenstein. In the following, we investigate the situation over a Cohen-Macaulay local ring with canonical module and generalize the result of Peskine and Szpiro \cite{PS}. Huneke in 1985, defined the concept of Strongly Cohen-Macaulay  and shows that when $R$ is Cohen-Macaulay and $I$ is Strongly Cohen-Macaulay then $R/J$ is Cohen-Macaulay. The theory then developed by introducing Sliding Depth conditions and the depth condition of the power of ideals and also extended from linkage theory to residual intersections and some other topics. 
	 
	 In  \cite[Proposition 1.1]{H}, Huneke has shown that $R/\Ann_R(I)$ is Cohen-Macaulay if $I$ is a strongly Cohen-Macaulay  ideal $R$. In section 2, we show that $R/\Ann_R(I)$ is indeed a maximal Cohen-Macaulay  $R$-module under the same assumption on $I$ without using of \cite[Proposition 1.1]{H}. Actually we give an another proof for \cite[Proposition 1.1]{H} by means of delta invariant (see Theorem \ref{70}).

	 In section 3,  we first show that if $I$ is a strongly Cohen-Macaulay ideal of $R$, then there exists an ideal $J$ of $R$ such that $ J\subseteq I $, $ R/J $ is Cohen-Macaulay and $\delta_{R/J}(H_i(I,R))=0$
	 for all  $ i> 0$ (see Proposition \ref{6}).
	 
	 Next, assume that $\omega_R$ is a canonical ideal of $R$ and  that $I$ is  Cohen-Macaulay ideal (not necessarily strongly Cohen-Macaulay) and show that  $\delta_{R/J}(H_s(I,\omega_R))=0$ for some Cohen-Macaulay ideal $J$, where $ J\subseteq I $ and $ s=\sup\{i:\  H_i(I,\omega_R)\neq0\} $ (see Theorem \ref{15}).

	\section{ 	Maximal Cohen-Macaulay-ness of annihilator of a strongly Cohen-Macaulay ideal}

	Assume that $R$ is a Cohen-Macaulay local ring and that $M$ is a finite $R$-module. A Cohen-Macaulay approximation of a finite $R$-module $M$ is a short exact sequence 
	\begin{equation}\label{e1}
	0\longrightarrow Y \longrightarrow X \overset{\varphi}{\longrightarrow} M \longrightarrow 0 
	\end{equation} such that $X$ is a maximal Cohen-Macaulay $R$-module and $Y$ is a finite $R$-module with finite injective dimension. We say that the sequence (\ref{e1}) is minimal if each endomorphism $ \psi $ of $ X $ with $ \varphi\circ \psi=\varphi $ is an automorphism of $ X $. If $R$ possesses a canonical module $ \omega_R$ then  a minimal Cohen-Macaulay approximation of $ M $ exists and is unique up to isomorphism (see \cite[Theorem 11.16]{LW},  \cite[Corollary 2.4]{HS}. It is well known that if (\ref{e1}) is a minimal Cohen-Macaulay approximation of $ M $, then  $ \delta_R(M) $ determines  the maximum rank of all free direct summands of $ X $, see \cite[Exercise 11.47]{LW} and \cite[Proposition 1.3]{D1}. As mentioned in \cite[Proposition 11.27]{LW}.  Also it can be shown that if $R$ is Cohen-Macaulay ring which admits a canonical module, then $\delta_R(M) $ is less than or equal to $ n $, where  there is an epimorphism $X\oplus R^{n}\longrightarrow M$ with $X $ a maximal Cohen-Macaulay module with no free direct summands (see \cite[Proposition 11.25]{LW} and  \cite[Proposition 4.8]{AA}). This definition of delta is used by Ding \cite{D}.
	
	We recall the basic properties of the delta invariant.
	\begin{prop}\label{a} 
	\cite[Corollary 11.28]{LW} and  \cite[Lemma 1.2]{ABIM}. Let $R$ be a Cohen-Macaulay local ring with canonical module. For finite $R$-modules $M $ and $N $ the following statements hold true.	
	\begin{enumerate} [\rm(i)]
	\item  $\delta_{R}(M\oplus N)=\delta_{R}\left( M\right)+\delta_{R}\left( N\right)  $.
	\item If  $M\longrightarrow N $ is a surjective homomorphism then $\delta_{R}\left( M\right)\geq \delta_{R}\left( N\right)$.  
	\item If $R$ is Gorenstein, then $\delta_{R}\left( k\right)=1$  if and only if $R$ is regular.
	\item  If $R$ is Gorenstein, then $\delta_{R}(M)=\mu_R(M) $ when  $\emph\pd_{R}(M) $ is finite.	
	\end{enumerate}
	\end{prop}
	Denote the $i$th homology module of the Koszul complex of $R$ with respect to a generating set of an ideal $I$ of $R$ by $H_i(I, R)$. In our discussions we use the following well known result.

	\begin{prop}\label{999} \cite[Proposition 1.6.11]{BH} Let $I$ be an ideal of $R$ and let $0\longrightarrow K\longrightarrow M\longrightarrow N\longrightarrow 0$ be an exact sequence of $R$-modules. Then there exists the long exact sequence 
			\begin{equation}
		\label{e2}
		\cdots\longrightarrow  H_i(I,K) \longrightarrow  H_i(I,M) \longrightarrow  H_i(I,N)\longrightarrow H_{i-1}(I,K) \longrightarrow H_{i-1}(I,M)\longrightarrow\cdots.
		\end{equation}
	\end{prop}
	
			An ideal $I$ is called {\it strongly Cohen-Macaulay} if $H_i(I, R)$ is Cohen--Macaulay for all $i$ (see \cite[Definition 3.1]{C0}). 
	Note that strongly Cohen-Macaulay-ness of an ideal $I$ is independent of its generators (see \cite[7.2.9]{FO}) for references).   
 
   We start this section by proving that $R/\Ann_R(I)$ is maximal Cohen-Macaulay  $R$-module by means of delta invariant. We prove this result by taking several steps. First we prove it in the case $R$ admits a canonical module.
 \begin{thm}\label{3}
 	Suppose that $ R $ is a Cohen-Macaulay local ring with a canonical module and that $ I (\neq0)$ is  a strongly Cohen-Macaulay ideal of $ R $.	If $ \emph\Ann_R(I)\neq0 $, then $ \delta_R(R/\emph\Ann_R(I))=0 $.	
 \end{thm}
 \begin{proof}
 	Our proof is by induction on $ \dim(R/I) $. Let $ \dim(R/I)=0 $. As  $ \Ann_R(I)\neq0 $, $ \gr_R(I, R)=0 $ which implies that $ \dim R=0 $ and so we have nothing to proof.
 	
 	Assume that $ \dim(R/I)=t>0 $ and statement has been proved for smaller than $ t $. Note that $\depth_R(H_i(I, R))=\dim(R/I)$ (see \cite[Remark 1.5]{H}).  Choose an element $ x\in R $ which is a non-zero divisor on $ R $ and on all Koszul homologies $ H_i(I,R) $. Set $ \overline{R}=R/xR $ and $ \overline{I} $ is image of $ I $ under homomorphism $R\longrightarrow R/xR $.  The exact sequence
 	$ 0\longrightarrow R\overset{x}{\longrightarrow} R\longrightarrow \overline{R}\longrightarrow0 $ gives the exact sequence 
 	$0\longrightarrow  H_i(I,R) \overset{x}{\longrightarrow} H_i(I,R) \longrightarrow  H_i(\overline{I},\overline{R})\longrightarrow 0 $.
 	Therefore $ H_i(\overline{I},\overline{R})\cong H_i(I,R)/xH_i(I,R) $ for all $ i\geq0 $. In particular, $\overline{R}/\overline{I}\cong (R/I)/x(R/I)$ and all $H_i(\overline{I},\overline{R})$
 	are Cohen-Macaulay. Hence, our induction hypothesis implies that $ \delta_{\overline{R}}(\overline{R}/(0 :_{\overline{R}} \overline{I}))=0 $.

 	To be precise, we assume that $I$ is generated by $n$ elements. Concentrating on the $n$th term of the Koszul complex,  we get
 	\[\begin{array}{rl}
 	(0 :_R I)/x (0 :_R I)
 	&\cong (0 :_{\overline{R}} \overline{I}) \\
 	&= (x :_R I)/xR\\
 	\end{array}\]    
 	On the other hand, As $xR\cap (0 :_R I)= x (0 :_R I)$,  we have 
 	\[\begin{array}{rl}
 	(0 :_R I)/x (0 :_R I)
 	&=(0 :_R I)/x\cap (0 :_R I) \\
 	&\cong xR+(0 :_R I)/xR\\
 	\end{array}\] 
 	The above natural isomorphisms imply that $(x :_R I)=xR+(0 :_R I)$. Hence

 	\[\begin{array}{rl}
 	\dfrac{R/(0 :_R I)}{x (R/(0 :_R I))}
 	&\cong R/( xR+(0 :_R I)) \\
 	&= R/(x :_R I) \\
 	&\cong \overline{R}/(0 :_{\overline{R}} \overline{I}).\\
 	\end{array}\]  
 	Note that,
 	\[\begin{array}{rl}
 	\delta_R(R/(0 :_R I))
 	&\leq \delta_{\overline{R}}(\dfrac{R/(0 :_R I)}{x (R/(0 :_R I))}) \  \text{  (see \cite[ Corollary 2.5]{K})}\\
 	&= \delta_{\overline{R}}(\overline{R}/(0 :_{\overline{R}} \overline{I}))\\
 	&=0\\
 	\end{array}\] 
 \end{proof}

 \begin{rmk}\label{74}
 	Let $ R $ be a Cohen-Macaulay local ring with a canonical module. Assume that $I$ is a strongly Cohen-Macaulay ideal of $R$ and $M$ is an $R$-module such that $(0 :_M I)\neq 0$. Then $\delta_R(M/(0 :_M I))=0$ if one of the following conditions hold.
 	\begin{enumerate} [\rm(a)]
 		\item $M$ is a maximal Cohen-Macaulay $R$-module.
 		\item  $\emph\dim R=1$ and $M$ is torsion-free $R$-module. 
 	\end{enumerate} 
 \end{rmk}
 \begin{proof}
 	It is enough to prove (a). As $(0 :_M I)\neq 0$, we have $\gr(I,M)=0$ so that
 	$$\dim_R(M/IM)
 	=\dim R\\
 	=\gr(I,R)+\dim_R(R/I)$$ because $M$ is maximal Cohen-Macaulay $R$-module.  
 	As $\dim_R(M/IM)\leq \dim_R(R/I)$, we have $\gr(I,R)=0$, i.e. $(0 :_R I)\neq 0$.
 	As $M$ is a homomorphic image of a finite free module, say 		 
 	$\overset{n}{\oplus}R\longrightarrow M$, one gets a surjective homomorphism $\overset{n}{\oplus}(R/(0 :_R I)) \longrightarrow M/(0 :_M I)$ and so 
 	\[\begin{array}{rl}
 	\delta_R(M/(0 :_M I))
 	&\leq \delta_R(\overset{n}{\oplus}(R/(0 :_R I)))\\
 	&\leq\overset{n}{\sum} \delta_R(R/(0 :_R I))=0.\\
 	\end{array}\]
 	by Theorem \ref{3}.
 \end{proof}

 We first show that $R/\Ann_R(I)$ is maximal Cohen-Macaulay when $R$ is Gorenstein.
 
 \begin{thm}\label{33}
 	Let $R$ be a  Gorenstein local ring and let $I$ be a  strongly Cohen-Macaulay  ideal  of $R$. Then $R/\emph\Ann_R(I)$ is a maximal Cohen-Macaulay $R$-module.
 \end{thm}
 \begin{proof}
 	We may assume that $ \Ann_R(I)\neq0 $. By Theorem \ref{3},   $ \delta_R(R/\Ann_R(I))=0 $. Therefore there exists a minimal Cohen-Macaulay approximation 
 	\begin{equation}
 	\label{e10}
 	0\longrightarrow Y\longrightarrow X \longrightarrow R/\Ann_R(I)\longrightarrow0
 	\end{equation}
 	of $ R/\Ann_R(I) $ such that $ X $ is a stable $ R $-module. We show, by induction on $t:=\id_R(Y) $, that if (\ref{e10}) is a minimal Cohen-Macaulay approximation such that $X$ is stable then $ Y=0 $.   If $ t=0 $ we have $\depth R=\id_R(Y)=0$. As $R$ is Gorenstein local ring, $\pd_R(Y)<\infty$. So we get 
 	$\pd_R(Y)
 	= \depth R - \depth_R(Y)=0$.
 	Therefore $Y$ is free $R$-module. As the sequence \ref{e10} splits we get a contradiction unless $ Y=0 $. 
 	
 	Now suppose that $ t>0 $ and the claim has been proved for smaller than $ t $. $\Ann_R(I)\neq0$ implies that $\gr(I, R)=0$ so we may choose  $x\in R $ which is a non-zero divisor on $R$ and on all $ H_i(I,R) $, $ i\geq0 $. Set $ \overline {R}=R/xR $, $ \overline{I}=I+xR/xR  $.  From the exact sequence
 	$ 0\longrightarrow R\overset{x}{\longrightarrow} R\longrightarrow \overline{R}\longrightarrow0 $
 	one has the exact sequence 
 	\begin{equation} \label{e30}
 	0\longrightarrow  H_i(I,R) \overset{x}{\longrightarrow } H_i(I,R) \longrightarrow  H_i(\overline{I},\overline{R})\longrightarrow 0,
 	\end{equation}
 	which shows that $ \overline I $ is also a strongly Cohen-Macaulay ideal of the Gorenstein ring $ \overline R $. As $ \Ann_R(I)\neq0$, (\ref{e30}) implies that $ \Ann_{\overline R}(\overline I)\neq0 $. By Theorem \ref{3}, $ \delta_{\overline R}(\overline R/ \Ann_{\overline R}(\overline I))=0 $. 
 	As $x$ is $R$-regular, it is $R/\Ann_R(I)$-regular and so from the exact sequence \ref{e10} we have the exact sequence
 	\begin{equation}
 	\label{e11}
 	0\longrightarrow Y/xY\longrightarrow X/xX \longrightarrow \overline R/\Ann_{\overline R}(\overline I)\longrightarrow 0.  
 	\end{equation}
 	which is a minimal Cohen-Macaulay approximation for $\overline{R}/\Ann_{\overline R}(\overline I)$ (see \cite[Lemma 5.1]{ADS}).
 	Note that a minimal Cohen-Macaulay approximation is unique up to isomorphism and $ \delta_{\overline R}(\overline R/ \Ann_{\overline R}(\overline I))=0 $, therefore $ X/xX $ is stable $ \overline R $ module. Notice that $\id(Y/xY)=t-1$  and by induction hypothesis we have $ Y/xY=0 $ and so the result follows.
 \end{proof}
 In order to present our main result, we need the following preparatory lemma.   
 \begin{lem}\label{333}
 	Assume that  $S \longrightarrow R$ is a surjective homomorphism of Cohen-Macaulay local rings with the same dimensions and that $I$ is a strongly Cohen-Macaulay ideal of $R$ such that $ \emph\Ann_R(I)\neq(0) $.  Set $I^c$ for the contraction of $I$ in $S$. Then there exists an element $z\in S$ which is either zero or is a non-zero divisor on $S$ such that  $I^c+zS/zS$ is a strongly Cohen-Macaulay ideal of the ring $\overline {S}=S/zS$.
 \end{lem}
 \begin{proof}
 	We proceed by induction on $t:=\depth_R( R/I)$. For  $t=0$ , one has 
 	\[\begin{array}{rl}
 	\depth_S(H_i(I^c,S))
 	&\leq \dim_S(H_i(I^c,S))\\
 	&=\dim_S(S/I^c)\\
 	& =\dim_R(R/I)=0, \\
 	\end{array}\] 
 	which means  $I^c$ is a strongly Cohen-Macaulay ideal of $S$.
 	
 	Assume that $ t>0 $ and the statement has been  proved for $ t-1 $. As $I$ is a strongly Cohen-Macaulay ideal of $R $,  $\dim S= t>0$ and there exists $x\in S $ such that $x $ is non-zero divisor on $S$,  $R$ and  on all $ H_i(I,R) $, $ i\geq0 $. Set $ \overline{R}=R/xR $, $ \overline{I}=I+xR/xR  $ and $\overline {S}=S/xS $ .  The exact sequence
 	$ 0\longrightarrow R\overset{x}{\longrightarrow} R\longrightarrow \overline{R}\longrightarrow0 $
 	implies the exact sequence 
 	\begin{center}
 		$0\longrightarrow  H_i(I,R) \overset{x}{\longrightarrow } H_i(I,R) \longrightarrow  H_i(\overline{I},\overline{R})\longrightarrow 0 $.
 	\end{center}
 	Therefore $H_i(\overline{I},\overline{R})$ is Cohen-Macaulay $R$-module and so $ \overline I $ is a strongly Cohen-Macaulay ideal of the Cohen-Macaulay ring $ \overline R $. Note that $\depth_{\overline {R}}(\overline {R}/\overline {I})=t-1 $ and there is the induced ring epimorphism  $\overline{S}\longrightarrow \overline {R} $ of Cohen-Macaulay  local rings with $\dim \overline {R}= \dim \overline {S}$ and that $\Ann_{\overline R}(\overline{I})\neq 0$. By induction hypothesis either $\overline{I}^c$  or $\overline{I}^c+\overline{y}\overline{S}/\overline{y}\overline{S}$ is a strongly Cohen-Macaulay ideal of $\overline{S}$ or $\overline{S}/\overline{y}\overline{S}$, respectively, for some $\overline{S}$-regular element $\overline{y}$. In case $\overline{I}^c$, i.e. $I^c+xR/xR$, is strongly Cohen-Macaulay ideal of $\overline{S}$, we put $z=x$. Assume that ${\overline {I}}^c+\overline {y}\overline{S}/\overline {y}\overline {S}$, i.e. $I^c+(x+y)S/(x+y)S$, is a strongly Cohen-Macaulay ideal of $S/(x+y)S$. Note $x, y$ is a regular sequence on the local ring $S$, it follows that $x+y$ is a regular element on $S$. By choosing $z=x+y$, the result follows. 
 \end{proof}
 
 Now we present our main result which give an another proof for the result of Huneke, i.e. with the same hypothesis, $R/\Ann_R(I)$ is a maximal Cohen-Macaulay $R$-module by means of the $\delta$-invariant. 
 \begin{thm}\label{70}
 	Let $R$ be a  Cohen-Macaulay local ring and let $I$ be a  strongly Cohen-Macaulay  ideal  of $R$. Then $R/\emph\Ann_R(I)$ is a maximal Cohen-Macaulay $R$-module. 	
 \end{thm}
 \begin{proof}
 	We may assume that $ \Ann_R(I)\neq(0)$.	As
 	$\widehat{H_i(I,R)}
 	\cong H_i(\widehat{I},\widehat{R})$, 
 	$\widehat{I}$ is a strongly Cohen-Macaulay $\widehat{R}$-module. Also maximal Cohen-Macaulay-ness of  $\widehat{R}/\Ann_{\widehat{R}}(\widehat{I})$ as $\widehat{R} $-module is equivalent to that  $R/\Ann_R(I)$ is a maximal Cohen-Macaulay $ R$-module. Therefore  we may assume that $R$ is a complete local ring. 
 	
 	Assume that $S\longrightarrow R$ is a ring epimorphism such that $S$ is a  Gorenstein local ring with $\dim R=\dim S$. Set  $I^c$ as the contraction of $I$. By Lemma \ref{333},  there exists an element $x\in S$ which is either zero or is a non-zero divisor on $S$ such that  $I^c+xS/xS$ is a strongly Cohen-Macaulay ideal of the Gorenstein ring $\overline {S}=S/xS$. By Theorem \ref{33}, $\overline {S}/ \Ann_{\overline {S}}(I^c+xS/xS)$ is a maximal Cohen-Macaulay $\overline {S}$-module. Note that $ \Ann_{S}(I^c)+ xS/xS \subseteq\Ann_{\overline {S}}(I^c+xS/xS)$ implies that 
 	\[\begin{array}{rl}
 	\dim \overline {S}
 	&=\dim_{\overline {S}}(\overline {S}/\Ann_{\overline {S}}(I^c+xS/xS))\\
 	& \leq \dim_{\overline {S}}\overline {S}/(\Ann_{S}(I^c)+ xS/xS) \\
 	& \leq \dim \overline {S}. \\
 	\end{array}\] 
 	
 	We also have
 	\[\begin{array}{rl}
 	\dfrac{S/\Ann_S(I^c)}{x(S/\Ann_S(I^c))}
 	&\cong \dfrac{S/\Ann_S(I^c)}{xS+\Ann_S(I^c)/\Ann_S(I^c)}\\
 	& \cong S/(xS+\Ann_S(I^c)) \\
 	& \cong \overline {S}/(xS+\Ann_S(I^c)/xS). \\
 	\end{array}\]
 	
 	Therefore $\dfrac{S/\Ann_S(I^c)}{x(S/\Ann_S(I^c))}$ is maximal Cohen-Macaulay $\overline{S}$-module. Hence    $S/\Ann_S(I^c)$ is a maximal Cohen-Macaulay $S$-module. As, there is a natural isomorphism $S/\Ann_S(I^c)\cong R/\Ann_R(I)$,  $R/\Ann_R(I)$  is a maximal Cohen-Macaulay $R$-module.
 \end{proof}
\begin{rmk}
	Theorem \ref{3} is a consequence of Theorem \ref{70}.
\end{rmk}
\begin{proof}
By Theorem \ref{70} we have $R/\emph\Ann_R(I)$ is a maximal Cohen-Macaulay $R$-module. As $R/\emph\Ann_R(I)$ has no free direct summand, we get $ \delta_R(R/\emph\Ann_R(I))=0 $. 
\end{proof}

	\section{Delta invariant of the Koszul homologies with respect to an strongly Cohen-Macaulay ideal}
	Assume that $ R $ is a Cohen-Macaulay local ring with a canonical module  and $I$ is a  strongly Cohen-Macaulay ideal of $R$. It is not true that
 $\delta_{R/I}(H_i(I,R))=0$ for all  $ i> 0$.
 For example, let $R=\bQ \llbracket x \rrbracket /(x^2)$ where $\bQ$ is the rational numbers. Assume that $I=(x)/(x^2)$. $I$ is a strongly Cohen Macaulay ideal of $R$ since $\dim_R((H_1(I,R)))=0=\dim_R((H_0(I,R)))$.  As $R/I$ is regular local ring and $H_1(I,R)\neq0$, we get $\delta_{R/I}(H_1(I,R))= \mu_{	R/I}(H_1(I,R))\neq0$ (see proposition \ref{a}).
 
Now we show that for any maximal regular sequence $J$ in	$I$, $\delta_{R/J}(H_i(I,R))=0$ for all  $ i> 0$.
\begin{prop} \label{6}
Assume that $ R $ is a Cohen-Macaulay local ring with a canonical module, and that $I$ is a  Strongly Cohen-Macaulay ideal of $R$, then for any maximal regular sequence $J$ in	$I$, $\delta_{R/J}(H_i(I,R))=0$ for all  $ i> 0$.
	
\end{prop}
\begin{proof} Set $J=(\underline{x})$, where $\underline{x}$ is a maximal regular sequence contained in $I$. If $\h(I)=0$, we have $\dim R/I=\dim R$ and so $H_i(I,R)$ are maximal Cohen-Macaulay $R$-modules. As $ I\subseteq \Ann_R(H_i(I,R)) $, we get $H_i(I,R)$ are stable $R$-module and result follows. Now assume that $\h(I)>0$.  As $\h(I)=\h((\underline{x}))$, we get 
	\[\begin{array}{rl}
	\dim R/(\underline{x}) &\geq \dim_{R/(\underline{x})}(H_i(I, R))\\
	&\geq\depth_{ R/(\underline{x})}(H_i(I, R))\\
	&= \depth_{ R}(H_i(I,R))\\
	&=\dim_{ R}(H_i(I,R))\\
	&=\dim_R(R/I) \\
	&=\dim R/(\underline{x}).
	\end{array}\] 
	Therefore all $H_i(I,R)$ are maximal Cohen-Macaulay $R/(\underline{x})$-modules. If $\delta_{R/(\underline{x})}(H_i(I, R))\not=0$ for some $i>0$, then $R/(\underline{x})$ is a direct summand of $H_i(I, R)$ which implies that $I=(\underline{x})$ and so $H_j(I, R)=0$ for all $j>0$, by \cite[Theorem 16.5]{HM}. As a result $R=(\underline{x})$ which is a contradiction. 
\end{proof}

We now assume that $I$ a (not necessarily strongly) Cohen-Macaulay ideal of  a Gorenstein ring $R$. It will be shown that the same claim as in Theorem \ref{6} holds true for the top non-zero Koszul homology of $I$. 

\begin{thm} \label{15}
	Let $R$ be a  Cohen-Macaulay  local ring  admits  canonical module $ \omega_R $. Suppose that $ I $ is a (not necessarily strongly) Cohen-Macaulay ideal of $ R $. Then there exists a Cohen-Macaulay ideal $ J\subseteq I $  such that $ \delta_{R/J}(\emph\Ext_R^{i}(R/I, \omega_R))=0 $ for all integer  $i\geq 0$.
	In particular,  $ \delta_{R/J}( H_s(I,\omega_R))=0 $, where $s=\mu_R(I)-\emph\h_R(I)$.
\end{thm}
\begin{proof}
	As $\Ext_R^{i}(R/I, \omega_R)=0$ for all $i \neq \h_R(I)$ (see  \cite [Exercise 3.1.24]{BH} and \cite[Theorem 17.1]{HM}), we have nothing to proof. For the case $i=\h_R(I)$, by induction on  $ \h_R(I)=t $. Let $ t=0 $, we choose $ J=0 $. By \cite[Proposition 3.3.3]{BH}, $ \Hom_R(R/I,\omega_R) $ is maximal  Cohen-Macaulay $ R $-module. As $\Hom_R(R/I,\omega_R) $ has no free direct summand, we get $ \delta_R(\Hom_R(R/I,\omega_R))=0 $.
	
	Now suppose that $ t> 0  $ and the claim has been proved for smaller than $ t $.  Choose a non-zero divisor $a\in I$ on  $R$ and set $ \overline {R}=R/aR $, $ \overline{I}=I/aR  $. Note that $\h_{\overline R}\overline{I}=t-1$ and $\overline{I}$ is a Cohen-Macaulay ideal of $\overline{R}$. By induction hypothesis,  there exists an ideal $J$ of $R$ containing $a$  such that $\overline J= J/aR\subseteq \overline I $, $\overline R/\overline J$ is Cohen-Macaulay ring satisfying  $ \delta_{ \overline R/\overline J}(\Ext_{\overline R}^{t-1}(\overline R/\overline I, \omega_{\overline R}))=0 $ .  Thus we have  we have
	\[\begin{array}{rl}
	\delta_{R/J}(\Ext_{R}^{t}( R/ I,  \omega_R))
	&= \delta_{\overline R/\overline J}(\Ext_{\overline R}^{t-1}(\overline R/\overline I, \omega_R/a\omega_R)) \\
	&=\delta_{\overline R/\overline J}(\Ext_{\overline R}^{t-1}(\overline R/\overline I, \omega_{\overline R}))\\
	&=0.
	\end{array}\]   
	The final claim is clear by  \cite[Theorem 1.6.16]{BH}.
\end{proof}

Let $D(-):=\Hom_R(-,\E_R(R/\fm))$ be the Matlis dual functor. Suppose that $ I $ is an $ \fm $-primary ideal of  $ R $. The following result is a corollary of Theorem \ref{15}, which shows that there exists a Cohen-Macaulay ideal $J$, $ J\subseteq I $,  such that $\delta_{R/J}(\E_{R/I}(R/\fm)) =0 $. Let $ I=\fm $ be a ideal of a Cohen-Macaulay local ring  $ R $ such that $\mu_R(I)=1$. We choose a Cohen-Macaulay ideal $J=\fm^2$. Then $ R/J $ is a Gorenstein, and  not regular local ring, equivalently  $\delta_{R/J}(\E_{R/\fm}(R/\fm)) =\delta_{R/J}(R/\fm)=0$ (see Proposition \ref{a}).

\begin{cor}
	Let $(R,\fm)$ be a  Cohen-Macaulay  local ring with a canonical module. Suppose that $ I $ is an $ \fm $-primary  ideal of $ R $. Then there exists a Cohen-Macaulay ideal $J$, $ J\subseteq I $,   such that  $\delta_{R/J}(\emph\E_{R/I}(R/\fm)) =0 $.
\end{cor}
\begin{proof}
	By Theorem \ref{15}, there is a Cohen-Macaulay ideal $J$ of $R$ with $J\subset I$ such that $ \delta_{R/J}(\Ext_R^{i}(R/I, \omega_R))=0 $ for all integer  $i\geq 0$.  In particular, for $i=d$, we have
	\[\begin{array}{rllll}
	\delta_{R/J}(\E_{R/I}(R/\fm))&=\delta_{R/J}(D( \Gamma_{\fm}(R/I)))&\\
	&= \delta_{R/J}(\Ext_R^{d}(R/I, \omega_R))\ \ \ \ &\text{(see \cite[Theorem 3.5.8]{BH})}\\
	& =0  &\text{(by Theorem \ref{15})}.
	\end{array}\]
	
\end{proof}

{\it Acknowledgments}. The authors would like to thank Hammid Hassanzadeh for his useful comments
which improvement the Theorem \ref{6}.

		\bibliographystyle{amsplain}
	
\end{document}